\documentclass[12pt]{amsart}
\usepackage[utf8]{inputenc}
\usepackage{mathrsfs}
\usepackage{amsrefs}
\usepackage{enumerate}
\usepackage{amssymb,amsmath,amsthm, amscd}
\usepackage[all]{xy}
\usepackage{tikz}

\theoremstyle{plain}
\newtheorem{thm}{Theorem}[section]
\newtheorem{lem}[thm]{Lemma}

\newtheorem{defn-lem}[thm]{Definition-Lemma}

\newtheorem{prop}[thm]{Proposition}

\theoremstyle{definition}
\newtheorem{defn}[thm]{Definition}

\newtheorem{rem}[thm]{Remark}

\def\dim{\operatorname{dim}}
\def\id{\operatorname{id}}
\def\dist{\operatorname{dist}}
\def \Image{\operatorname{Image}}

\def\QT{\operatorname{QT}}
\def\T{\operatorname{T}}
\def\her{\operatorname{her}}
\def\Trdim{\operatorname{Trdim}}
\def\nuc{\operatorname{nuc}}
\newcommand{\mc}{\mathcal}

\numberwithin{equation}{section}

\begin{document}
\title [ II ]
       {On permanence of regularity properties II}

\begin{abstract}
In this article, we study the permanence of topological and algebraic dimension type properties of simple unital $C\sp*$-algebras.  When a pair of unital $C\sp*$-algebras $(A, B)$ is associated by a $*$-homomorphism $\phi: A\to B$ which is  tracially sequentially-split by order zero, we show that under suitable assumptions on either $A$ or $B$, or  both $A$ and $B$
\begin{itemize}
\item tracial $m$-comparison passes from $B$ to $A$;
\item tracial $m$-almost divisibility passes from $B$ to $A$;
\item tracial nuclear dimension less than $m$ passes from $B$ to $A$.
\end{itemize}   
\end{abstract}

\author { Hyun Ho \, Lee}

\address {Department of Mathematics\\
        University of Ulsan\\
         Ulsan, 44610 South Korea 
          }
\email{hadamard@ulsan.ac.kr}


\date{}
\maketitle
\section{Introduction} 
\medskip
\noindent 
The classification program for simple nuclear $C^*$-algebras, initiated by Elliott, has reached a significant milestone with the completion of the classification for algebras with finite nuclear dimension \cite{CETWW, EGLN}. Central to this achievement is the Toms-Winter conjecture, which posits that for a simple, separable, unital, and nuclear $C^*$-algebra, the conditions of $\mathcal{Z}$-stability, strict comparison of positive elements, and finite nuclear dimension are equivalent \cite{Ro, W}.\\ 

\noindent 
In the study of these regularity properties, a recurring theme is their permanence under structural operations such as crossed products, inclusions, and tensor products. A highly successful approach to permanence was introduced by Barlak and Szabó \cite{BS}, who utilized the notion of a sequentially split $*$-homomorphism $\phi: A \to B$ to show that regularity properties pass from $B$ to $A$. However, in many important dynamical settings—particularly those involving the tracial Rokhlin property or large subalgebras \cite{P}—the splitting maps are often only available in a tracial sense.\\ 

\noindent
In the first part of this study \cite{LO}, we began an investigation into the tracial analogues of these permanence results. We focused on pairs $(A, B)$ connected by a *-homomorphism that is tracially sequentially-split. While we successfully demonstrated the permanence of $\mathcal{Z}$-absorption and strict comparison, the third pillar of the Toms-Winter conjecture—nuclear dimension—remained elusive. The primary technical obstacle was the inability of transferring the tracially small part to the norm part.  Later,  in \cite{LO1} we refined our framework by requiring the tracial splitting to be implemented by a c.p.c. order zero map in the sense of Winter and Zacharias \cite{WZ}. Order zero maps are essentially $^*$-homomorphisms from the cone of an algebra, making them the ''correct" class of maps for preserving orthogonality and Cuntz-semigroup data. By requiring the map to be tracially sequentially-split by order zero, we develop a unified conceptual framework that encompasses both:
\begin{enumerate}
\item Actions of finite groups with the weak tracial Rokhlin property, and
\item Inclusions of $C^*$-algebras $P \subset A$ with a conditional expectation $E: A \to P$ has the weak tracial Rokhlin property
\end{enumerate}
\medskip

\noindent 
While the finiteness of nuclear dimension is a rigid structural property, tracial nuclear dimension is specifically designed to manage the tracially negligible remainder. Since this objective aligns precisely with our framework, the permanence of this  invariant is natural and expected consequence. Our main results show that if $B$ is a simple unital $C^*$-algebra with tracial $m$-comparison, tracial $m$-almost divisibility, or tracial nuclear dimension at most $m$, then these properties pass to $A$ whenever $A$ is tracially sequentially-split in $B$ by an order zero map. It is worth noting that tracial $m$-comparison and $m$-almost divisibility serve as algebraic measurements of the structural regularity of a $C^*$-algebra. While the permanence of these algebraic invariants provides a necessary foundation, the proof for the tracial nuclear dimension is highly technical; it requires a sophisticated coordination of order zero splitting and orthogonal lifting within the tracial ultrapower to preserve topological complexity while avoiding dimension inflation. Thus this work effectively bridges the gap between the algebraic and topological facets of our framework, completing the tracial regularity program initiated in \cite{LO1}.\\

\noindent 
This article is organized as follows. In Section 2, we review the necessary background on Cuntz subequivalence and the structure of order zero maps \cite{WZ}. In Section 3, we prove our main permanence theorems, providing a comprehensive tracial counterpart to the classical Barlak-Szabó framework and further solidifying the stability of regularity properties within the Elliott program.    
          
\section{Preliminaries}\label{S:2}

In this section, we review the fundamental tools and notation necessary for our study. We begin by recalling the definition and basic properties of the Cuntz subequivalence, which provides the primary order structure on the Cuntz semigroup. Furthermore, we establish the technical background for completely positive contractive (c.p.c.) maps of order zero, as their structural rigidity is essential for the lifting arguments employed in our permanence results.
\begin{defn}
Let $A$ be a $C^*$-algebra. For positive elements $a, b \in A_+$, we say $a$ is \textit{Cuntz subequivalent} to $b$, denoted $a \lesssim b$, if there exists a sequence $\{v_n\}_{n=1}^{\infty} \subset A$ such that 
\[ \lim_{n \to \infty} \| v_n^* b v_n - a \| = 0. \]
\end{defn}

\begin{defn}[Dimensional Independence in $M_{\infty}(A)$]
Let $M_{\infty}(A) = \bigcup_{n=1}^{\infty} M_n(A)$ be the algebraic limit of matrix algebras over $A$ under the upper-left inclusions $a \mapsto \begin{pmatrix}  a & 0 \\ 0 & 0 \end{pmatrix}$. For $a \in M_n(A)_+$ and $b \in M_m(A)_+$, the relation $a \lesssim b$ is defined as follows:
\begin{enumerate}
    \item Let $k \ge \max(n, m)$. Let $\iota_n: M_n(A) \to M_k(A)$ and $\iota_m: M_m(A) \to M_k(A)$ be the natural inclusions.
    \item We define $a \lesssim b$ if and only if $\iota_n(a) \lesssim \iota_m(b)$ holds in the $C^*$-algebra $M_k(A)$.
\end{enumerate}
This relation is independent of the choice of $k \in \mathbb{N}$.
\end{defn}

\begin{prop}[Rectangular Matrix Characterization]
For $a \in M_n(A)_+$ and $b \in M_m(A)_+$, the following are equivalent:
\begin{itemize}
    \item $a \lesssim b$ in $M_{\infty}(A)$.
    \item There exists a sequence of rectangular matrices $\{v_k\}_{k=1}^{\infty} \subset M_{m,n}(A)$ such that $v_k^* b v_k \to a$ in norm.
    \item For every $\varepsilon > 0$, there exists $v \in M_{m,n}(A)$ such that $(a - \varepsilon)_+ = v^* b v$.
\end{itemize}
\end{prop}

\begin{defn}[Cuntz Equivalence]
Two positive elements $a, b \in M_{\infty}(A)_+$ are said to be \textit{Cuntz equivalent}, denoted $a \sim b$, if 
\[ a \lesssim b \quad \text{and} \quad b \lesssim a. \]
This is an equivalence relation on the set $M_{\infty}(A)_+$. 
\end{defn}
The following lemma is borrowed from \cite[Lemma 1.4]{P} and summarizes well-known facts about Cuntz subequivalence. We denote by $A^{+} $ the unitization of a $C\sp*$-algebra $A$. 
\begin{lem} \label{lem:cuntz_properties}
Let $A$ be a $C^*$-algebra.
\begin{enumerate}
    \item Let $a, b \in A_+$. If $a \in \overline{bAb}$, then $a \lesssim b$.
    \item Let $a \in A_+$ and let $f: [0, \infty) \to [0, \infty)$ be a continuous function such that $f(0) = 0$. Then $f(a) \lesssim a$.
    \item Let $a \in A_+$ and let $f: [0, \|a\|] \to [0, \infty)$ be a continuous function such that $f(0) = 0$ and $f(\lambda) > 0$ for $\lambda > 0$. Then $f(a) \sim  a$.
    \item Let $c \in A$. Then $c^*c \sim cc^*$.
    \item Let $a \in A_+$ and let $u \in A$ be a unitary. Then $uau^* \sim  a$.
    \item Let $c \in A$ and let $\alpha > 0$. Then $(c^*c - \alpha)_+ \sim  (cc^* - \alpha)_+$.
    \item Let $v \in A$. Then there is an isomorphism $\phi: \overline{v^*vAv^*v} \to \overline{vv^*Avv^*}$ such that for every positive element $z \in \overline{v^*vAv^*v}$, we have $z \sim  \phi(z)$.
    \item Let $a \in A_+$ and let $\varepsilon_1, \varepsilon_2 > 0$. Then $((a - \varepsilon_1)_+ - \varepsilon_2)_+ = (a - (\varepsilon_1 + \varepsilon_2))_+$.
    \item Let $a, b \in A_+$ satisfy $a \lesssim b$ and let $\delta > 0$. Then there exists $v \in A$ such that $v^*v = (a - \delta)_+$ and $vv^* \in \overline{bAb}$.
    \item Let $a, b \in A_+$. Then $\|a - b\| < \varepsilon$ implies $(a - \varepsilon)_{+} \lesssim  b$.
    \item Let $a, b \in A_+$. The following are equivalent:
    \begin{enumerate}
        \item $a \lesssim b$.
        \item $(a - \varepsilon)_{+} \lesssim b$ for all $\varepsilon > 0$.
        \item For every $\varepsilon > 0$, there exists $\delta > 0$ such that $(a - \varepsilon)_{+} \lesssim (b - \delta)_+$.
    \end{enumerate}
    \item Let $a, b \in A_+$. Then $a + b \lesssim a \oplus b$.
    \item Let $a, b \in A_+$ be orthogonal (i.e., $ab = 0$). Then $a + b \sim a \oplus b$.
    \item Let $a_1, a_2, b_1, b_2 \in A_+$. If $a_1 \lesssim a_2$ and $b_1 \lesssim b_2$, then $a_1 \oplus b_1 \lesssim a_2 \oplus b_2$.
\end{enumerate}
\end{lem}
\begin{defn}[The Semigroup $W(A)$]
The \textit{Cuntz semigroup} of $A$, denoted by $W(A)$, is the set of Cuntz equivalence classes of positive elements in $M_{\infty}(A)$:
\[ W(A) := M_{\infty}(A)_+ / \sim = \left\{ \langle a \rangle : a \in M_{\infty}(A)_+ \right\}. \]
The set $W(A)$ is equipped with the following structure:
\begin{enumerate}
    \item \textbf{Addition:} For $\langle a \rangle, \langle b \rangle \in W(A)$, the addition is defined via the direct sum:
    \[ \langle a \rangle + \langle b \rangle := \langle a \oplus b \rangle, \]
    where $a \oplus b = \text{diag}(a, b) \in M_{n+m}(A)$ for $a \in M_n(A)$ and $b \in M_m(A)$.
    \item \textbf{Partial Order:} $W(A)$ is an ordered abelian semigroup where the order is induced by the subequivalence relation:
    \[ \langle a \rangle \le \langle b \rangle \iff a \lesssim b. \]
\end{enumerate}
\end{defn}

\begin{prop}[Well-definedness]
The operations defined above are well-defined:
\begin{itemize}
    \item If $a_1 \sim a_2$ and $b_1 \sim b_2$, then $a_1 \oplus b_1 \sim a_2 \oplus b_2$.
    \item The addition is associative and commutative within $W(A)$.
    \item The zero element is $\langle 0 \rangle$, the class of the zero matrix.
\end{itemize}
\end{prop}

\begin{defn}\label{D:dimensionfunction}
For a unital $C\sp*$-algebra $A$, we denote by $\QT(A)$ the set of normalized 2-quasitraces on $A$. When $A$ is stably finite, we define $d_{\tau}: M_{\infty}(A)_{+} \to [0, \infty)$  by $d_{\tau}(a)=\lim_{n \to \infty} \tau(a^{1/n})$ for $a \in M_{\infty}(a)_{+}$ where $\tau \in \QT(A)$. 
\end{defn}

\begin{prop}
Let $A$ be a stably finite unital $C\sp*$-algebra. Then $d_{\tau}$ in Definition \ref{D:dimensionfunction} is well-defined on $W(A)$. That is, if $a, b \in M_{\infty}(A)_{+}$ such that $a \sim b$, then $d_{\tau}(a)=d_{\tau}(b)$. 
\end{prop}
From now on, we will be concerned with only the class of exact $C\sp*$-algebras $A$,  and thus $\QT(A)=\T(A)$ the set of tracial states of $A$ in this case. \\  
\begin{defn}(Winter-Zacharias)
Let $A$ and $B$ be $C^*$-algebras. A completely positive contractive (c.p.c.) map $\phi: A \to B$ is said to be of \textit{order zero} if it preserves orthogonality. That is, for any $a, b \in A_+$,
\[ a \cdot b = 0 \implies \phi(a) \cdot \phi(b) = 0. \]
\end{defn}
We shall abbreviate a c.p.c order zero map as an order zero map without confusion. The following theorem is about the structure of an order zero map. 

\begin{thm} \cite[Corollary 4.1]{WZ}\label{T:semiprojective}
Let $A$ and $B$ be $C\sp*$-algebras, and $\phi:A \to B$ an order zero map. Then the map given by $\varrho_{\phi} (\id_{(0,1]}\otimes a):=\phi(a)$ induces a $*$-homomorphism $\varrho_{\phi}:C_0((0,1])\otimes A \to B$. Conversely, any $*$-homomorphism $\varrho:C_0((0,1])\otimes A \to B$ induces an order zero map 
$\phi_{\varrho}(a):=\varrho(\id_{(0,1]}\otimes a)$.\\
These mutual  assignments yield a canonical correspondence between the space of c.p.c. order zero maps from $A$ to $B$  and the space of $*$-homomorphisms from $C_0((0,1])\otimes A $ to $B$.
\end{thm}

Throughout the paper, we fix a free ultrafilter $\omega$ on $\mathbb{N}$ and recall that  
$l^{\infty}(\mathbb{N}, A)=\prod_{n=1}^{\infty} A$ denotes  the $C\sp*$-algebra of  all bounded function from $\mathbb{N}$  to  $A$. We define a closed ideal of $l^{\infty}(\mathbb{N}, A)$ as follows: 
\[c_{\omega}(\mathbb{N}, A)=\{(a_n)_n \in l^{\infty}(\mathbb{N}, A)  \mid \lim_{\omega}\|a_n \|=0 \}. \]

Then we denote by $A^{\omega}=\prod_{n=1}^{\infty} A /c_{\omega}(\mathbb{N}, A)$ the ultrapower $C\sp*$-algebra of $A$  with respect to the filter $\omega$ that is  equipped with the norm  $\|a\|= \lim_{\omega} \|a_n\| $ for $a=[(a_n)_n] \in A^{\omega}$. In addition, we denote by $\pi_{\omega}$ the canonical quotient map from $\prod_{n=1}^{\infty} A$  onto $A^{\omega}$. Note that we can embed $A$ into $A^{\omega}$ as constant sequences, and we call $A^{\omega} \cap A'$ the central ultrapower algebra of $A$. 

Since $ \|a\|=\lim_{\omega}\|a_n\|$ for $a=[(a_n)_n] \in A^{\omega}$, for any nonzero projection $p=[(p_n)] \in A^{\omega}$ we may assume that each $p_n$ is a nonzero projection in $A$. Similarly, for any nonzero positive element $a \in A^{\omega}$ we can represent it as $a=[(a_n)_n]$ such that each $a_n$ is a nonzero positive element in $A$ and  $\{ \|a_n\| \mid n=1,2, \dots \}$ is uniformly away from zero.

A unified conceptual framework encompassing both the weak tracial Rokhlin property for finite group actions and inclusions of $C\sp*$-algebras  has been suggested in \cite{LO1} motivated by \cite{BS, LO}. 
\begin{defn} \cite[Definition 3.1]{LO1}\label{D:sequentiallysplitbyorderzero}
 A $^*$-homomorphism $\phi:A \to B$ is called tracially sequentially-split by order zero map, if for every positive nonzero element $z \in A^{\omega}$  there exist a c.p.c map of order zero  $\psi: B \to A^{\omega}$ and a nonzero contractive positive element $g\in A^{\omega}\cap A'$ such that 
\begin{enumerate}
\item $\psi(\phi(a))=ag$ for all $a\in A$, 
\item $1_{A^{\omega}} -g \lesssim z$ in $A^{\omega}$.
\end{enumerate}
We call $\psi$ a tracial approximate left inverse.  
\end{defn}

 We use the diagram below  to symbolize that $\phi$ is tracially sequentially-split by order zero map with its tracial approximate left inverse $\psi$;
\begin{equation*}\label{D:diagram}
\xymatrix{ A \ar[rd]_{\phi} \ar@{-->}[rr]^{\iota} && A^{\omega} &\\
                          & B \ar[ur]_{\psi: \,\text{order zero.}} }
\end{equation*}

The following proposition provides a mechanism for transferring properties from the matrix-extended level $M_n(A)$ back to the underlying $C^*$-algebra $A$. 
\begin{prop}\cite[Proposition 3.3]{LO1}\label{P:amplification}
 Let $A$ be a unital inifinite dimensional simple  $C\sp*$-algebra and $B$ be a unital $C\sp*$-algebra. Suppose that $\phi:A \to B$ is tracially sequentially-split by order zero map and unit preserving. Then its amplification $\phi\otimes \id_{M_n}: A\otimes M_n \to B\otimes M_n$ is tracially sequentially-split by order zero map for any $n\in \mathbb{N}$.  
\end{prop}

Having established the requisite background on Cuntz subequivalence and the technical properties of c.p.c. order zero maps, we are now positioned to investigate their global implications. The following section is devoted to proving that regularity properties of $B$ are inherited by $A$ via the tracial splitting map. We begin with $m$-comparison, which serves as a prototype for the stability of our framework.

\section{Main results}\label{S:3}

In this section, we prove three permanence results, which are the main results of this paper; first, tracial $m$-comparison; second, tracial $m$-almost divisibility; and third, tracial nuclear dimension,  all of which are related to nuclear dimension $C\sp*$-algebra $A$, denoted by $\dim_{\nuc}A$, of Winter and Zacharias \cite{WZ2010}. These all show that the properties possessed by $B$ is transferred to $A$ through $\phi:A \to B$ that is tracially sequentially-split by order zero map. 
   
\begin{defn}\cite[Definition 2.1-(i)]{W}
Let $A$ be a separable, simple, unital $C\sp*$-algebra and $m\in \mathbb{N}$. 
It is said that $A$ has $m$-comparison of positive elements, if for any nonzero positive contractions $a, b_0, \dots, b_m \in M_{\infty}(A)$ we have 
\[ a \lesssim b_0 \oplus \dots \oplus b_m\] 
whenever 
\[ d_{\tau} (a) < d_{\tau}(b_i) \]  
for every $\tau \in \QT(A)$ and $i=0, \dots, m$. Based upon \cite[Proposition 2.3]{W}, we also say that $A$ has tracial $m$-comparison if $A$ has $m$-comparison. 
\end{defn}
\begin{rem}
It is known  in \cite[Theorem 1]{Le} that If $A$ is a separable $C\sp*$-algebra with $\dim_{\nuc}A \le  m$, then $A$
has $m$-comparison. Thus $m$-comparison can be thought as an algebraic expression of topological dimension. 
\end{rem}
 \begin{thm}\label{T:m-comparison}
Let $A$ and $B$ be  separable stably finite simple unital $C\sp*$-algebras and $\phi: A \to B $ a $*$-homomorphism which is tracially sequentially-split by order zero. If $B$ satisfies  tracial $m$-comparison in the sense of W. Winter, then so does $A$.   
\end{thm}
\begin{proof}
In view of Proposition \ref{P:amplification}, we may assume that $a, b_0, \dots b_m \in A$. Suppose that for each $i=0, 1, \dots, m$ \[d_{\tau} (a)< d_{\tau}(b_i)   \quad \forall \tau \in T(A).\]  We want to conclude that 
\[ a \lesssim b_0 \oplus b_1 \oplus \cdots \oplus b_m. \] 
Since $\displaystyle d_{\tau}(\phi(a)) < d_{\tau}(\phi(b_i))$ for $i=0, \dots, m$ and $B$ satisfies $m$-comparison, it follows that 
\[ \phi_{m+1} (a) \lesssim \phi_{m+1} (b_0\oplus \cdots \oplus b_m).\]
If at least one of $b_i$'s is purely positive in $A$, then $ \oplus_{i=0}^m b_i$ is purely positive in $M_{m+1}(A)$.  Thus, 
by \cite[Proposition 2.18]{LO} we have $a \lesssim b_0 \oplus \cdots \oplus b_{m}$. 
If none of $b_i$'s  is purely positive, then we may assume that each $b_i$ is Cuntz equivalent to a projection $q_i$ and split two cases: \\
Case 1: If $a$ is not purely positive, i.e., $a$ is Cuntz equivalent to a projection $p$ in $A$, note that for each $i=0, \dots, m$
the function $\tau \ni T(B) \to d_{\tau} (\phi(b_i)) -d_{\tau}(\phi(a))$ is continuous and positive. Since $T(B)$ is w$^*$- compact, it follows that  
\begin{equation}
\rho_i=  \inf_{\tau \in T(A)} (d_{\tau}(p) - d_{\tau}(q_i))> 0.
\end{equation}
Then by Lemma \cite[Lemma 2.21]{LO} we can find $c_i \in W(A)^{+}$ such that $d_{\tau} (p) < d_{\tau}(b_i)- \rho<  d_{\tau}(c_i)$ and $c_i \lesssim b_i$. \\
Now we are in a situation that for each $i=0, 1, \dots, m$ \[d_{\tau} (a)< d_{\tau}(c_i)   \quad \forall \tau \in T(A),\] where $c_i$'s are purely positive and $c_i \lesssim b_i$ for all $i$.  Again by \cite[Proposition 2.18]{LO} 
we have 
\[ a \lesssim c_0 \oplus \cdots \oplus c_m \lesssim b_0 \oplus \cdots \oplus b_m.\]

Case 2: If $a$ is purely positive, we want to show that for any $\epsilon >0$ 
\[(a-\epsilon)_{+} \lesssim b_0 \oplus \cdots \oplus b_m.\]
Given $\epsilon >0$ consider a function $f:[0, \infty) \to [0,1]$ such that $f(\lambda)>0$ for $\lambda \in (0, \epsilon)$ and $f(\lambda)=0$ for $\lambda \in \{ 0\} \cup [\epsilon, \infty)$. Then $f(a) \neq 0$ and $f(a) \perp (a-\epsilon)_{+}$.
Therefore, $\rho=\int \{d_{\tau} (f(a))\mid \tau \in T(A) \}$ is strictly positive. 
\[
\begin{split}
d_{\tau}((a-\epsilon)_{+})+ \rho &\le d_{\tau}((a-\epsilon)_{+})+ d_{\tau}(f(a)) \\
&=d_{\tau}((a-\epsilon)_{+}+ f(a)) \le d_{\tau}(a) < d_{\tau}(b_i).
\end{split}
\]
Again, we use \cite[Lemma 2.21]{LO} to find $c_i \in W(A)^{+}$ such that $c_i \lesssim b_i$ and $d_{\tau} (c_i)> d_{\tau}(b_i) -\rho > d_{\tau}((a-\epsilon)_{+})$. Consequently, 
\[d_{\tau}(\phi(a-\epsilon)_{+}) < d_{\tau}(\phi (c_i)) \quad \text{for all $\tau \in T(B)$}\] with $c_i$ being purely positive.  By the same argument in the above 
\[ (a-\epsilon)_{+} \lesssim  c_0 \oplus \cdots \oplus c_m \lesssim b_0 \oplus \cdots \oplus b_m\] 

\end{proof} 

\begin{defn}\cite[Definition 2.5-(ii)]{W}
Let $A$ be a unital exact $C\sp*$-algebra and $m\in \mathbb{N}$. It is said that $A$ is \emph{tracially $m$-almost divisible}, if for any positive contraction $a \in M_{\infty}(A)$, $\epsilon >0$ and $0\neq k \in \mathbb{N}$, there is a c.p.c order zero map
\[ \psi: M_k \to \her(a) \subset M_{\infty}(A)\]
such that 
\[\tau(\psi(\mathbf{1_k})) \ge \frac{1}{m+1} \cdot \tau(a) -\epsilon \]
for all $\tau \in T(A)$. 
\end{defn}
\begin{rem}
It is also known in \cite[Proposition 3.7]{W} that if $A$ is a separable, simple, non elementary and unital $C^*$-algebra with $\dim_{\nuc}A \le
m$, then $A$ has tracial $\widetilde{m}$-almost divisibility for some $\widetilde{m} \in \mathbb{N}$. Thus tracial $m$-almost divisibility can be thought as another algebraic expression of topological dimension.
\end{rem}
\begin{thm}\label{T:divisibility}
Let $A$ and $B$ be  simple exact unital $C\sp*$-algebras and $\phi: A \to B $ a $^*$-homomorphism which is tracially sequentially-split by order zero. Then if $B$ is \emph{tracially $m$-almost divisible}, so is $A$. 
\end{thm}
\begin{proof}
In view of Proposition \ref{P:amplification}, for $a \in A\setminus \{0\}$, $\epsilon >0$ and $0\neq k \in \mathbb{N}$, it is enough to show that there is a c.p.c. order zero map $\nu:M_k \to her(a) \subset A$ such that $\tau(\nu(\mathbf{1_k})) \ge \frac{1}{m+1} \tau(a) -\epsilon$ for all $\tau \in T(A)$. We consider a triple $\phi(a) \in B$, $\epsilon >0$ and $0\neq k \in \mathbb{N}$. Since $B$ is tracially $m$-almost divisible, there is a c.p.c order zero map $\mu: M_k \to her(\phi(a)) \subset B$ such that 
\begin{equation}\label{E:star}
\tau(\mu(\mathbf{1_k})) \ge \frac{1}{m+1} \tau(\phi(a)) -\frac{\epsilon}{4} 
\end{equation}
for all $\tau \in T(B)$. For a singleton $\{a \}$, by the assumption there is a c.p.c order zero map $\psi: B \to A^{\omega}$ such that $\psi (\phi (a))= ag$ where $\psi(\mathbf{1_B})=g \in A^{\omega}\cap A'$ and $\widetilde{\tau}(a(1-g)) < \epsilon/4$ for all $\widetilde{\tau} \in T(A^{\omega})$.  Consider $\psi\circ \mu : M_{k} \to A^{\omega}$ which is c.p.c. order zero. If we denote by $\pi: \prod_{i=1}^{\infty} A \to A^{\omega}$ the natural quotient map, we can write $\psi \circ \mu = \pi \circ \prod_{i=1}^{\infty} \kappa_i $ where $\kappa_i:M_k \to A$ is a c.p.c order zero map for each $i$ by \cite[Lemma 2.1]{TWW}. Note that $\psi \circ \mu: M_k \to \her(a) \subset A^{\omega}$, and thus $\kappa_i: M_k \to \her(a) \subset A$ for each $i$.  
\\
For $\tau \in T(A)$, consider $\tau^{\omega}: A^{\omega} \to \mathbb{C}$ which is given by $\tau^{\omega} ([a_n])=\lim_{n \to \omega} \tau(a_n)$. Then 
\[ \tau^{\omega}((\psi \circ \mu)(\mathbf{1_k}))= \lim_{n \to \omega} \tau (\kappa_n(\mathbf{1_k})) \]
But, by (\ref{E:star})
\[ 
\begin{split}
\tau^{\omega} \circ \psi (\mu(\mathbf{1_k}) ) &\ge \frac{1}{m+1} (\tau^{\omega} \circ \psi ) (\phi(a)) - \frac{\epsilon}{4} \\
\lim_{n \to \omega} \tau (\kappa_n (\mathbf{1_k})) &\ge \frac{1}{m+1} \tau^{\omega}(ag)- \frac{\epsilon}{4} \\
& \ge \frac{1}{m+1} (\tau^{\omega}(a) - \tau^{\omega}(a(1-g))) -\frac{\epsilon}{4} \\
& \ge \frac{1}{m+1} \tau(a) - \frac{\epsilon}{4} -\frac{\epsilon}{4}\\
& \ge \frac{1}{m+1} \tau (a) -\frac{\epsilon}{2}  
\end{split}
\]
On the other hand, there is an infinite set $\mc{U} \in \omega$  such that $ \displaystyle \tau(\kappa_n(\mathbf{1_k}) ) > \lim_{n \to \omega} \tau(\kappa_n(\mathbf{1_k})) - \frac{\epsilon}{2}$ for all $n \in. \mc{U}$.  Pick $n$ in $\mc{U}$, then $\kappa_n: M_k \to A$ satisfies the following;
\[ \tau(\kappa_n(\mathbf{1_k})) > \frac{1}{m+1} -\epsilon.\]  So we take $\nu=\kappa_n$ and finish the proof. 
\end{proof}
While the previous results focused on the algebraic order of the Cuntz semigroup, we now turn our attention to the topological complexity of the algebras. The permanence of tracial nuclear dimension requires a more sophisticated lifting argument. Specifically, we must ensure that the finite-dimensional approximations of $B$ can be translated to $A$ without losing their orthogonality. This is achieved through the orthogonal lifting techniques developed in Lemma \ref{L:Kirchberg} and Lemma \ref{L:orthogonallift}. 

\begin{defn}\cite[Definition 7.6]{FL}
Let $A$ be a unital simple $C\sp*$-algebra and let $n \in \mathbb{N}\cup \{0\}$. It is said that $id_A$ has tracial nuclear dimension no more than $n$, if for any finite subset $\mc{F} \subset A $, any $\epsilon >0$, and any $a \in A_{+}\setminus \{0\}$, there exist a finite dimensional $C\sp*$-algebra $F$, a c.p.c. map $\alpha: A \to F$, a nonzero piecewise contractive $n$-decomposable c.p. map $\beta: F \to A$, and a c.p.c map $\gamma: A \to A\cap \beta(F)^{\perp}$, such that 
\begin{enumerate}
\item $ \|x - (\gamma(x)+ \beta \circ \alpha (x)) \| < \epsilon$ for all $x\in \mc{F}$, and 
\item $\gamma(\mathbf{1_A}) \lesssim a$ in $A$. 
\end{enumerate}
In this case, we write $\Trdim_{\nuc} \id_A = \Trdim_{\nuc} A \le n$. 
\end{defn}
An important tool for working with ultrapowers are re-indexing arguments, which allow one to find elements of the ultrapower exactly satisfying some given condition if one can find elements of the ultrapower which approximately satisfy the condttion for any given tolerance.  
\begin{lem}(Kirchberg $\epsilon$-test)\cite[Lemma 1.8]{CE} \cite[Lemma A.1]{K}\label{L:Kirchberg}
Let $X_1, X_2, \dots $ be a sequence of non-empty sets, and  for each $k, n \in \mathbb{N}$ let $f^{(k)}_n: X_n \to [0,\infty)$ be a function. Define $f^{(k)}_{\omega}: \prod_{n=1}^{\infty} X_n \to [0, \infty]$ by $f^{(k)}_{\omega} ((s_n)_{n=1}^{\infty}):= \lim_{n \to \omega} f^{(k)}_n (s_n)$  for $(s_n) \in \prod_{n=1}^{\infty}X_n$. Suppose that for all $m \in \mathbb{N}$ and $\epsilon>0$, there exist $(s_n)_{n=1}^{\infty} \in \prod_{n=1}^{\infty}X_n$ with $f^{(k)}_{\omega}((s_n)) < \epsilon$ for $k=1,\dots, m$. Then there exists $(t_n)_{n=1}^{\infty} \in \prod_{n=1}^{\infty}X_n$ such that $f^{(k)}_{\omega} ((t_n))=0$ for all $k \in \mathbb{N}$. 
\end{lem}

\begin{lem}(orthogonal lift of a c.p.c map)\label{L:orthogonallift}
Let $A$ be a separable nuclear unital $C\sp*$-algebra, $F$ a finite dimensional $C\sp*$-algebra and $\beta^{\omega}: F \to A^{\omega}$ a nonzero piecewise contractive  $m$-decomposable c.p. map. Put $D^{\omega}:=A^{\omega}\cap \beta^{\omega}(F)^{\perp}$. If $\Gamma^{\omega}: A \to D^{\omega}$ is c.p.c. then there exists a sequence of c.p.c. maps $\Gamma_n: A \to A$ such that 
\begin{enumerate}
\item $ [\Gamma_n(x)]=\Gamma^{\omega}(x)$ for all $x\in A$;
\item $ \Gamma_n(A) \subset A\cap \beta_n(F)^{\perp}$ for all $n$  
\end{enumerate}   
where $\beta^{\omega}=[\prod_{n=1}^{\infty} \beta_n]$ and $\beta_n: F \to A$ is a c.p. map  for each $n$. 
\end{lem}
\begin{proof}
The idea of the proof is to combine Lemma \ref{L:Kirchberg} and the cut down by a support projection which plays like the identity in the multiplication. The detail is as follows. 
Write $\Gamma^{\omega}=[(\Gamma_n)]$ and $\beta^{\omega}=[(\beta_n)]$. From the given condition, 
\begin{equation} 
\lim_{n \to \omega} \|\Gamma_n(a)\beta_n(e^i_{jk})\|=0 
\end{equation}
where for all $a \in A$, $\{e^{i}_{jk}\}$ are matrix units for $F^i$ for $F=\oplus_{i=1}^m F^i$. 
To employ Kirchberg $\epsilon$-test to find a sequence of c.p. maps, we need to set $X_n=\{ \text{c.p.c maps from $A$ to $A$}\}$ for each $n$.  Let $S_A=\{ a_j  \mid  j=1,2, \dots\}$ be a countable dense subset of $A$, $B_F=\{ e^{i}_{k,l} \mid i, j, k\}$ a basis of  $F$. To satisfy two conditions (1) and (2), we need two functions measuring the lifting and orthogonality and merge them into one. For $k, n \in \mathbb{N}$ we define 
 \[ f^{(k)}_n(\psi):= \|\psi(a_k)-\Gamma_n(a_k) \| + \sum_{b \in B_F } \| \psi(a_k)\beta_n(b)\| \] Note that the latter is a finite sum. \\
 Now, given $m\in \mathbb{N}$ and $\varepsilon >0$, we set $h_n=\beta_n(\mathbf{1_F}) \in A$. Then $H:=\her (h_n)$ and    $g_n:=(f_\sigma(h_n))^{1/2} \in H$, where
\[
f_\sigma(t)=
\begin{cases}
0,& t=0,\\
\dfrac{(t-\sigma)_+}{t},& t > 0.
\end{cases}
\] and  $\sigma$ will be chosen later.  

    Let $\Gamma^1_{n}(x) = (1 - g_n) \Gamma_n(x) (1 - g_n)$. We estimate the distance from the initial lift $\Gamma_n$: 
  For $x\in A$ we have \[ \Gamma_n(x^*)\Gamma_n(x) \le \|\Gamma_n(1)\| \Gamma_n(x^*x).\]
    Taking the norm both sides, we have 
    \[ \|\Gamma_n(x) \|^2 \le \| \Gamma_n(1)\|\|x\|^2 \]
    since $\Gamma_n$ is a c.p.c. map. Then 
    \begin{align*}
    \|\Gamma^1_{n}(x) - \Gamma_n(x)\| &\le \|g_n \Gamma_n(x)\| + \|(1-g_n) \Gamma_n(x) g_n\| \\
    &\le \|x\| \|\Gamma_n(1_A)^{1/2} g_n\| + \|1-g_n\| \|x\| \|\Gamma_n(1_A)^{1/2} g_n\| \\
    &\le 2 \|x\| \|g_n \Gamma_n(1_A) g_n\|^{1/2}
    \end{align*}
    Since $\Gamma^{\omega} \perp \beta^{\omega}$, we have $\lim_{n \to \omega} \| \Gamma_n(1_A) \beta_n(1_F) \| = 0$. Because $g_n$ is supported on $\beta_n(1_F)$, it follows that $\lim_{n \to \omega} \|g_n \Gamma_n(1_A) g_n\| = 0$.  Thus  
    \[ \lim_{n \to \omega}\|\Gamma^1_{n}(x) - \Gamma_n(x)\|=0. \]

    For any $b \in F$, the second term is controlled by the interaction of $(1-g_n)$ and the image of $\beta_n$:
    \begin{align*}
    \|\Gamma_{n, \perp}(x) \beta_n(b)\| &= \|(1-g_n) \Gamma_n(x) (1-g_n) \beta_n(b)\| \\
    &\le \|\Gamma_n(x)\| \|(1-g_n) \beta_n(b)\|
    \end{align*}
    By the definition of $g_n = f(\beta_n(\mathbf{1_F}))$ and the property of the functional calculus $(\sqrt{t} - \sqrt{t-\sigma}) \le \sqrt{\sigma}$, we have:
    \[ \|(1-g_n) \beta_n(b)\| \le \sqrt{\sigma} \|b\|. \]
  If we choose $\sigma>0$ such that $\sigma \le \left( \frac{\varepsilon}{\sum_{l=1}^M \|b_l \|}\right)^2$, 
    \[ \lim_{n \to \omega} \sum_{l=1}^M \|\Gamma^1_n(a_j)\beta_n(b_l)\| \le \sqrt{\sigma} \sum_{ l =1}^{M} \|b_l\| < \varepsilon. \]
 Thus, $f^{(k)}_{\omega}((\Gamma^1_n)) < \varepsilon$ for $k=1, \dots, m$. By the Kirchberg $\varepsilon$-test, there exists a sequence $(\Gamma_n)_{n=1}^{\infty}$ of c.p.c. maps such that $f^{(k)}_{\omega}((\Gamma_n)) = 0$ for all $k$, satisfying both the lifting and the limit orthogonality requirements. \\
 
To achieve exact orthogonality at the sequence level, we perform a final compression as follows:
Let $e_n \in A^{**}$ be the support projection of $h_n$ which comes from $1_{(0, 1]}(h_n)$ by Borel functional calculus. Since $F$ is finite dimensional, $(1-e_n)$ acts as the strict annihilator of the image of $\beta_n$.
Note that $H_n:=\her^{\perp}(h_n)$ is also a hereditary $C\sp*$-subalgebra of $A$ corresponding to the open projection $(1-e_n)$, in particular norm closed \cite{P} and $\beta_n(F) \subset \her(b_n)$. 
For each $n$, let $\mc{F}_n:=\{a_1, \dots, a_n\} \subset  S_A$.  

Since \[ \lim_{n \to \omega} \dist(\Gamma_n(a_i), H_n)=0, \]
 we consider $y_{n, i} \in H_n$ such that $\| \Gamma_n(a_i) -y_{n, i}\| < \delta_n$ where $\delta_n \to 0$ along the ultrafilter $\omega$.   We let $(u_{n,k})_{k=1}^{\infty}$ be a quasi-central approximate unit in $H_n$.   Then 
 \[\lim_{k \to \infty}\max_{1\le  i \le n}\{ \|u_{n,k}y_{n,i}u_{n,k} -y_{n,i}\|\} =0\]
 If we take $k=k(n)$ large enough  such that $\|u_{n,k(n)} y_{n,i} u_{n,k(n)}- y_{n,i} \| < \delta_n$ for $1 \le i \le n$, 
 \[\begin{split}
 \max_{i \le n} \|u_{n,k(n)}\Gamma_n(a_i)u_{n,k(n)}  - \Gamma_n(a_i)\| & \max_{i \le n}\le \| u_{n,k(n)}\Gamma_n(a_i)u_{n,k(n)} - u_{n,k(n)}y_{n,i}u_{n,k(n)} \|  \\
  &+ \max_{i \le n}\| u_{n,k(n)}y_nu_{n,k(n)}- y_{n,i} \|+ \max_{i \le n}\|y_{n,i} -\Gamma_n(a_i) \| \\
  &\le 3 \delta_n.
  \end{split}
  \] Thus 
    \[\lim_{n\to \omega} \max_{1 \le i \le n} \| u_{n,k(n)}\Gamma_n(a_i)u_{n,k(n)} - \Gamma_n(a_i) \|=0. \]  Note that $u_{n, k(n)}\Gamma_n(x) u_{n, k(n)} \in H_n$,  and thus $u_{n, k(n)}\Gamma_n(x) u_{n, k(n)}\beta_n(F)=0$.  By a standard $3\epsilon$-argument over the dense sequence, we obtain a sequence of c.p.c. maps $\tilde{\Gamma}_n(x) := u_{n,k(n)} T_n(x) u_{n,k(n)}$ mapping $A$ into $H_n$ which satisfies $[\tilde{\Gamma}_n] = [T_n] = \Gamma^{\omega}$.
\end{proof}

\begin{thm}\label{T:Trdim}
Let $A$ be a nuclear simple unital infinite dimensional $C\sp*$-algebra and $B$ be a simple nuclear unital $C\sp*$-algebras and $\phi: A \to B $ a $^*$-homomorphism which is tracially sequentially-split by order zero. If $\Trdim_{\nuc} B \le m $, then $\Trdim_{\nuc} A \le m$. 
\end{thm}
\begin{proof} 
Given a finite set $\mc{F} \subset A$, $\epsilon>0$, $a_0 \in A_{+} \setminus \{ 0\}$, consider two nonzero positive elements $a_1, a_2 \in \her(a_0)$ such that $a_1 \perp a_2$ and $a_1 \sim a_2$.  consider a function $h$ defined by 
\[
 h(x)=\begin{cases}
 0 \quad & \text{if $0 \le x \le \eta$}\\
 \frac{x-\eta}{\eta} \quad & \text{if $\eta \le x \le 2\eta$}\\
 1 \quad &\text{if $x \ge 2 \eta$}
  \end{cases}
\] 
where $\eta >0$ is decided later.  Let $z:= h(a_1) \in A^{+}$. 
Then using the assumption we can take a c.p.c. order zero map $\psi: B \to A^{\omega}$  and $\psi(\mathbf{1_B})=g \in A^{\omega}\cap A'$ such that $\psi(\phi(a))=ag$,  $1- g \lesssim z \lesssim a_1$ since $z \in \her(a_1)$. 
We also consider $0< \delta < \epsilon$ which will be decided later.  Let $\mc{G}=\phi(\mc{F}) \subset B$, $\epsilon/6$, and $b_0=f(\phi(a_2))$ where $f\ge 0$ is a continuous function such that $f(0)=0, f(t)> 0$  for $t>0$. 
Since $\Trdim_{\nuc} B \le m$, there exist a finite dimensional $C\sp*$-algebra $F_0$, a c.p.c. map $\alpha_B: B \to F$, a nonzero piecewise $m$-decomposable c.p. map $\beta_B: F \to B$, a c.p.c. map $\gamma_B: B \to B\cap \beta(F)^{\perp}$, such that  
 \begin{enumerate}
\item $ \|y - (\gamma_B(y)+ \beta_B \circ \alpha_B (y) )\| < \epsilon/6$ for all $y\in \mc{G}$, and 
\item $\gamma_B(\mathbf{1_B}) \lesssim b_0$ in $A$. 
\end{enumerate}
Combining $\phi$ and $\psi$ with these maps, we define $\alpha: A \to F$ by $\alpha:= \alpha_B \circ \phi$,  $\beta^{\omega}: F \to A^{\omega}$ by $\beta^{\omega}:=\psi \circ \beta_B$, and $\gamma^{\omega}: A \to A^{\omega}$ by $\gamma^{\omega}:=\psi \circ \gamma_B \circ \phi$. Since $\psi$ is of order zero, it follows that \[\gamma^{\omega}(A) \perp \beta^{\omega}(F). \]

From (1), for any  $y \in \mc{G}$  
\begin{equation}
\| \psi(y) -(\psi(\gamma_B (y))+ \psi (\beta_B(\alpha_B(y)))) \| < \frac{\epsilon}{6}.
\end{equation}
So, for  any $x \in \mc{F}$
\begin{equation}
\| \psi(\phi(x)) -(\psi(\gamma_B (\phi(x)))+ \psi (\beta_B(\alpha_B(\phi(x))))) \| < \frac{\epsilon}{6}
\end{equation}
Thus, for  any $x \in \mc{F}$
\begin{equation}
\| xg -(\gamma^{\omega} (x)+  \beta^{\omega} (\alpha(x)) ) \| < \frac{\epsilon}{6}
\end{equation}
If we write $x=xg+ x(1-g)$, we have now the new remainder $\gamma^{\omega}(x) + x(1-g)$ replacing the role of $\gamma$ in the original definition. The problem is then the map $x \mapsto \gamma^{\omega}(x)+ x(1-g) \in A^{\omega}$ may not be orthogonal to $\beta^{\omega}(F)$. Let $b:=\beta^{\omega}(\mathbf{1_F}) \in A^{\omega}_{+} \setminus \{0\}$ and  $e=s(b)$ the support of $b$ in $(A^{\omega})^{**}$. It is well known that $\her(b)= A^{\omega}\cap e(A^{\omega})^{**}e$    and $\her^{\perp}(b)= A^{\omega}\cap (1-e)(A^{\omega})^{**}(1-e)= \overline{\{x \in A^{\omega} \mid xb=bx=0 \}}^{\| \,\|}$ due to Akemann and Pedersen. The latter is also a hereditary $C\sp*$-subalgebra.  
We note that $b$ is in the range of $\psi$, which means that $b \in \her(\psi(\mathbf{1_B}))=\her(g)$, and thus  $b=gc^{**}g$ for some $c^{**} \in (A^{\omega})^{**}$.  It follows that $b(1-g)=(1-g)b=0$. In fact, for every $x\in A$ $x(1-g)=(1-g)x \in \her^{\perp}(b)$ and in particular $1-g \in \her^{\perp}(b)$.  If we let $f_{\eta}(t)=\max\{0, t-\eta\}$, then $q:= f_{\eta}(1-g) \in \her^{\perp}(b)$ and $\| (1-g) -q\| \le \eta$. If we define $R^{\omega}(x):=q^{1/2}xq^{1/2}=qx$ for $x\in A$ and set $\eta= \frac{\epsilon}{ 6 \max\{ \|x\| \mid x\in \mc{F}\}}$, we can easily deduce that 
\begin{equation}
\| x - (\gamma^{\omega}(x) + R^{\omega}(x) + (\beta^{\omega}\circ \alpha)(c)) \| < \frac{\epsilon}{3}
\end{equation}
for all $x\in \mc{F}$. 
Moreover,  it is not hard to show that $R^{\omega}(x)\in \her^{\perp}(b)$ and $\Image( \beta^{\omega} )\subset \her(b)$. Therefore $R^{\omega}(A) \perp \beta^{\omega}(F)$.\\

On the other hand,  since $\psi$ is of order zero, 
\[
\gamma^{\omega}(\mathbf{1_A})=\psi(\gamma_B(\mathbf{1_B})) \lesssim \psi(f(\phi(a_2))) = f(a_2 g) \in \her(a_2). 
\]
Consequently, 
\begin{equation}\label{E:Cuntz1}
\gamma^{\omega}(\mathbf{1_A}) \lesssim a_2.
\end{equation}
Also, 
\begin{equation}\label{E:Cuntz2}
R^{\omega}(\mathbf{1_A})=q \le1-g \lesssim a_1
\end{equation}
So if we define $\Gamma^{\omega}:= \gamma^{\omega}+ R^{\omega}$, we obtain the following; 
\begin{equation}\label{E:condition(1)}
\|.x -(\Gamma^{\omega}(x)+ \beta^{\omega} \circ \alpha(x))  \| < \frac{\epsilon}{3}
\end{equation} 
\begin{equation}\label{E:condition(2)}
\Gamma^{\omega}(A) \perp \beta^{\omega}(F)
\end{equation}
\begin{equation}\label{E:condition(3)}
\Gamma^{\omega}(\mathbf{1_A}) \lesssim a_1 +a_2 \lesssim a_0  
\end{equation}
in $A^{\omega}$.\\  
Since $A$ nuclear, we can lift $\Gamma^{\omega}$ to $\Gamma: A \to \prod_{i=1}^{\infty} A$ and write $\Gamma= \prod_{n=1}^{\infty} \Gamma_n $ where $\Gamma_n: A \to A$ is a c.p. map by Choi-Effros. In addition, since $F$ is also nuclear, we can lift $\beta^{\omega}$ to $\widetilde{\beta}: F \to \prod_{n=1}^{\infty} A$ and write $\widetilde{\beta}=\prod_{n=1}^{\infty} \beta_n$ where $\beta_n: F \to A$ is a c.p.c map.    When we write $F=\oplus_{i=1}^m F_i $, the restriction map $\beta^{\omega}|_{F_i}$ is c.p.c order zero map for each $i$, we can write the c.p.c map $\beta_n=\sum_i \beta_{n, i}$ where $\beta_{n, i}: F_i \to A$ is also c.p.c order zero map. Combining (\ref{E:condition(2)}) with  Lemma \ref{L:orthogonallift}, we assume that $\Gamma_n(A) \perp \beta_n(F)$ for all $n$.\\  
\noindent
 From (\ref{E:condition(3)}), for each $k \in \mathbb{N}$, there exist $r^{(k)}=[(r^k_n)_{n=1}^{\infty}] \in A^{\omega}$ such that 
 \[  \| (r^{(k)})^* a_0 r^{(k)} - \Gamma(\mathbf{1_A}) \|  < \frac{1}{2^k} \quad \text{ in. $A^{\omega}$}. \] 
 Equivalently, 
 \[ \lim_{n \to \omega} \| (r^k_n)^* a_0 r^k_n - \Gamma_n (\mathbf{1_A}) \|_A < \frac{1}{2^k} \] 
 Therefore, $S_k:= \{j \in \mathbb{N}\mid \| (r^k_j)^* a_0  r^k_j - \Gamma_j (\mathbf{1_A})\|_A < \frac{1}{2^k}  \} \in \omega$. Note that $S_k$ has infinitely many elements for each $k$. \\
\noindent
For each $x \in \mc{F}$, from (\ref{E:condition(1)}), 
\[ \lim_{n\to \omega} \| x -(\Gamma_n(x)+ \beta_n \circ \alpha (x) )  \| < \frac{\epsilon}{3}. \]
Hence the set $U_x=\{n \mid \| x -(\Gamma_n(x)+ \beta_n \circ \alpha (x) )  \| < \frac{\epsilon}{3} \} \in \omega$. 
 Let
\[
U:=\bigcap_{x\in\mc{F}} U_x \in\omega.
\]
\noindent
For an arbitrary $\sigma >0$ (to be chosen later) and choose $k\in\mathbb N$ such that $2^{-k}<\sigma$. Since $U\in\omega$ as well, we have $U\cap S_k\in\omega$,
so we may pick $n\in U\cap S_k$. Then
\[
\|(r^{(k)}_{n})^*a_0 r^{(k)}_{n}-\Gamma_{n}(\mathbf{1_A})\|<2^{-k}<\sigma.
\]

Since $(r^{(k)}_{n})^*a_0 r^{(k)}_{n}\precsim a_0$ in $A$, the standard
perturbation-to-Cuntz-sub\-equivalence estimate yields
\[
(\Gamma_{n}(\mathbf{1_A})-\delta_1)_+\ \precsim\ (r^k_{n})^*a_0 r^k_{n}
\ \precsim\ a_0.
\]
Moreover, 
\begin{equation}\label{E:normestimate}
\| x -(\Gamma_{n}(x)+ \beta_{n} \circ \alpha (x) )  \| < \frac{\epsilon}{3}. 
\end{equation}
Now we set $\beta=\beta_n$, $\Gamma=\Gamma_n$ for our choice $n$. 

Now we need to deform $\Gamma$ to $\gamma$ such that $\gamma(\mathbf{1_A})= (\Gamma(\mathbf{1_A})-\sigma)_{+}$ and $\| \gamma(x) -\Gamma(x)\| < \dfrac{\epsilon}{3}$ for all $x \in \mc{F}$.  The following method might be known to expects;   
 First,  let 
\[H:=\her(h) \subset D_n=A\cap \beta_n(F)^{\perp} \] where $\Gamma(\mathbf{1_A})=h$.  Let $k_\sigma:=(f_\sigma(h))^{1/2} \in H$, where
\[
f_\sigma(t)=
\begin{cases}
0,& t=0,\\
\dfrac{(t-\sigma)_+}{t},& t > 0.
\end{cases}
\]
We define
\[
 \gamma(x):=k_{\sigma}\,\Gamma(x)\,k_{\sigma} \quad (x\in A).
\]

(1) Since $\gamma$ is obtained from $\Gamma$ by two-sided multiplication with a contraction,
it is completely positive and contractive and $k_{\sigma} \in H \subset D_n$ implies that $\gamma(A) \subset D_n$. 

\smallskip
(2) Since $k_\sigma$ commutes with $h$, we have
\[
\gamma(1_A)=k_\sigma\,h\,k_\sigma=g_\sigma(h)\,h.
\]
For  $t\ge 0$ one has $g_\sigma(t)\,t=(t-\sigma)_+$, hence by functional calculus
\[
\gamma(1_A)=(h-\sigma)_+.
\]

\smallskip
(3) Fix $x\in A$.
Let $\Gamma(x)=V^*\pi(x)V$ be a Stinespring dilation, so that $h=\Gamma(1_A)=V^*V$.
For $\varepsilon>0$ define the bounded operator
\[
W_\varepsilon:=V\,(h+\varepsilon)^{-1/2}.
\]
Then
\[
W_\varepsilon^*W_\varepsilon
=(h+\varepsilon)^{-1/2}V^*V(h+\varepsilon)^{-1/2}
=(h+\varepsilon)^{-1/2}h(h+\varepsilon)^{-1/2}
\le 1,
\]
so $\|W_\varepsilon\|\le 1$.
Set
\[
y_{x,\varepsilon}:=W_\varepsilon^*\pi(x)W_\varepsilon,
\]
which satisfies $\|y_{x,\varepsilon}\|\le \|x\|$.

Define
\[
s_\varepsilon:=h^{1/2}(h+\varepsilon)^{-1/2}\in C^*(h),\qquad 0\le s_\varepsilon\le 1.
\]
A direct computation gives
\[
s_\varepsilon\,\Gamma(x)\,s_\varepsilon
= h^{1/2}y_{x,\varepsilon}h^{1/2}.
\]
Moreover, $s_\varepsilon\to 1$ strongly on $\overline{\mathrm{Ran}(h)}$, and hence
\[
\|s_\varepsilon\Gamma(x)s_\varepsilon-\Gamma(x)\|\xrightarrow[\varepsilon\downarrow 0]{}0.
\]

Since $k_\sigma$ commutes with $h$ and
\[
k_\sigma h^{1/2}=(h-\sigma)_+^{1/2}
\]
by functional calculus, we have
\[
\gamma(x)
=k_\sigma\,h^{1/2}y_{x,\varepsilon}h^{1/2}k_\sigma
=(h-\sigma)_+^{1/2}y_{x,\varepsilon}(h-\sigma)_+^{1/2}.
\]

Let $d:=h^{1/2}-(h-\sigma)_+^{1/2}$. Then
\[
h^{1/2}y_{x,\varepsilon}h^{1/2}
-(h-\sigma)_+^{1/2}y_{x,\varepsilon}(h-\sigma)_+^{1/2}
= d\,y_{x,\varepsilon}\,h^{1/2}
+(h-\sigma)_+^{1/2}y_{x,\varepsilon}\,d.
\]
Taking norms and using $\|h^{1/2}\|\le 1$ and $\|(h-\sigma)_+^{1/2}\|\le 1$,
\[
\|h^{1/2}y_{x,\varepsilon}h^{1/2}
-(h-\sigma)_+^{1/2}y_{x,\varepsilon}(h-\sigma)_+^{1/2}\|
\le 2\|d\|\,\|y_{x,\varepsilon}\|.
\]
For  $t\in[0,1]$ one has
\[
0\le \sqrt{t}-\sqrt{(t-\sigma)_+}\le \sqrt{\sigma},
\]
hence $\|d\|\le \sqrt{\sigma}$ by functional calculus.
Since $\|y_{x,\varepsilon}\|\le \|x\|$, we obtain
\[
\|h^{1/2}y_{x,\varepsilon}h^{1/2}
-(h-\sigma)_+^{1/2}y_{x,\varepsilon}(h-\sigma)_+^{1/2}\|
\le 2\sqrt{\sigma}\,\|x\|.
\]

Letting $\varepsilon\downarrow 0$ yields
\[
\|\Gamma(x)-\gamma(x)\|\le 2\sqrt{\sigma}\,\|x\|,
\]
as required.
Now if $M:=\max\{ \|x\| \mid x \in \mc{F}\}> 0$, and we choose $0< \sigma < \frac{\varepsilon^2}{144M^2}$, then

\begin{equation}\label{E:perturbation}
\|\Gamma(x)-\gamma(x)\|<\frac{\varepsilon}{6}
\end{equation}
for all $x \in \mc{F}$. 
Combining two estimates (\ref{E:normestimate}) and (\ref{E:perturbation}) we obtain 
\[ 
\| x - (\gamma(x) + \beta \circ \alpha (x))\| < \frac{\epsilon}{2}< \epsilon
\]
for all $x \in \mc{F}$. 
Thus for given a finite set $\mc{F}$, $\varepsilon>0$, a nonzero positive element $a_0 \in A$, we have shown that there exist a finite dimensional $C\sp*$-algebra $F$, a c.p.c. map $\alpha: A \to F$, a nonzero piecewise contractive $m$-decomposable c.p. map $\beta: F \to A$, and a c.p.c map $\gamma: A \to A\cap \beta(F)^{\perp}$, such that 
\begin{enumerate}
\item $ \|x - (\gamma(x)+ \beta \circ \alpha (x)) \| < \epsilon$ for all $x\in \mc{F}$, and 
\item $\gamma(\mathbf{1_A}) \lesssim a_0$ in $A$. 
\end{enumerate}
Hence $\Trdim_{\nuc}A \le m$. 
\end{proof}
\begin{rem}
The conclusion of the proof of Theorem \ref{T:Trdim} signifies the completion of the tracial permanence program initiated in [8]. By establishing the inheritance of tracial nuclear dimension, we have shown that all three fundamental regularity properties—$m$-comparison, $m$-almost divisibility, and nuclear dimension—behave consistently within the framework of tracial sequential-splitness. In the broader context of the Toms-Winter conjecture, these results confirm that the tracial versions of the regularity pillars are robust under the unified framework covering both weak tracial Rokhlin actions and inclusions.
\end{rem}

\end{document}